\documentclass[10pt]{article}
\usepackage[a4paper, tmargin=4cm, bmargin=4.0cm, lmargin=3.5cm, rmargin=3.5cm, textheight=24cm, textwidth=16cm]{geometry}
\usepackage{setspace}
\usepackage{amsfonts}
\usepackage{epsfig}
\usepackage{latexsym}
\usepackage{amsthm}
\usepackage{amsmath}
\usepackage{amssymb}
\usepackage{array}
\usepackage{enumerate}
\usepackage{enumitem}
\usepackage{graphicx}
\usepackage{breqn}
\usepackage{thmtools, thm-restate}

\newtheorem{theorem}{Theorem}[section]

\newtheorem{identity}{Identity}[section]
\newtheorem{observation}{Observation}[section]
\newtheorem{corollary}{Corollary}[section]

\newtheorem{lemma}{Lemma}[section]
\newtheorem{remark}{Remark}[section]

\newlist{notes}{enumerate}{1}
\setlist[notes]{label=Note: ,leftmargin=*}

\title{\textbf{On the distance spectrum of the Kronecker product of distance regular graphs}}
\author{ Priti Prasanna Mondal, Fouzul Atik
 \\ \small Department of Mathematics, SRM University-AP, Andhra Pradesh 522240, India.\\ \small e-mail: pritiprasanna1992@gmail.com, fouzulatik@gmail.com, }

\date{}

\begin{document}
\maketitle{}
\begin{abstract}
Consider two simple graphs, $\mathcal{G}_{1}$ and $\mathcal{G}_{2}$, with their respective vertex sets $V(\mathcal{G}{1})$ and $V(\mathcal{G}_{2})$. The Kronecker product, denoted as $\mathcal{G}_{1} \otimes \mathcal{G}_{2}$, forms a new graph with a vertex set $V(\mathcal{G}_{1}) \times V(\mathcal{G}_{2})$. In this new graph, two vertices, $(x,y)$ and $(u,v)$, are adjacent if and only if $xu$ is an edge in $\mathcal{G}_{1}$ and $yv$ is an edge in $\mathcal{G}_{2}$. While the adjacency spectrum of this product is known, the distance spectrum remains unexplored. This article determines the distance spectrum of the Kronecker product for a few families of distance regular graphs. We find the exact polynomial, which expresses the distance matrix $D$ as a polynomial of the adjacency matrix, for two distance regular graphs, Johnson and Hamming graphs. Additionally, we present families of distance integral graphs, shedding light on a previously posted open problem given by Indulal and Balakrishnan in (\emph{AKCE International Journal of Graphs and Combinatorics, 13(3):230–234, 2016}). \\

Keywords: Distance matrix, Kronecker product, Distance regular graphs, Distance spectrum, Integral spectrum.
\end{abstract}

\section{Introduction}
Let $\mathcal{G}=(V,E)$ be a simple connected graph with the vertex set $V=\{1, 2, \ldots, n\}$ and edge set $E$. The distance matrix of $\mathcal{G}$ is denoted by $D(\mathcal{G})$ and defined as $D(\mathcal{G})=(d_{ij})_{n\times n}$, where $d_{ij}$ is the length (number of edges) of the shortest path between vertices $i$ and $j$. The spectrum of a matrix is the set of all eigenvalues of that matrix. Our notation for eigenvalues and their multiplicities is denoted as $\lambda^{(\alpha)}$, where $\alpha$ represents the multiplicity of the eigenvalue $\lambda$. The vector $j_n$ is defined as a column vector of order $n$ with all entries equal to 1. The spectrum of the distance matrix plays a crucial role in various fields of science and engineering. Interest in this topic began in the 1970s with the publication of the paper by Graham and Pollack \cite{grah}. In the mentioned article, the authors demonstrated the eigenvalue inertia of a tree, tackled issues within data communication systems, and established that the determinant of a tree's distance matrix remains unaffected by the tree's structure. The distance matrix and distance spectra find numerous applications, whether in implicit or explicit forms. These applications span various fields, including the design of communication networks \cite{elzin, grah}, network flow algorithms \cite{fred}, graph embedding theory \cite{deza, edel}, and molecular stability \cite{zhou}. Balaban et al. \cite{bala} proposed utilizing the distance spectral radius as a molecular descriptor.

The distance spectra are extremely beneficial in many fields of engineering and science. In the literature, only a few families of graphs with known full-distance spectra exist. Ruzieh and Powers \cite{ruzi} determined all the distance spectra of the path graph $\mathcal{P}_{n}$. Graovac et al. \cite{grao} have provided all the distance eigenvalues of the cycle graph $\mathcal{C}_{n}$.

Many significant classes of graphs can be constructed using graph products. Let $\mathcal{G}$ and $\mathcal{H}$ be two graphs with vertex sets $V(\mathcal{G})={ x_{1}, x_{2}, \ldots, x_{n}}$ and $V(\mathcal{H})={ y_{1}, y_{2}, \ldots, y_{m}}$ respectively. In most cases, the graph product of two graphs $\mathcal{G}$ and $\mathcal{H}$ is a new graph whose vertex set is $V(\mathcal{G}) \times V(\mathcal{H})$, the Cartesian product of the sets $V(\mathcal{G})$ and $V(\mathcal{H})$. The adjacency of two vertices $(x_{i}, y_{j})$ and $(x_{r}, y_{s})$ in a product graph is mainly defined by the adjacency, equality, or non-adjacency of $x_{i}$ and $x_{r}$ in $\mathcal{G}$ and that of $y_{j}$ and $y_{s}$ in $\mathcal{H}$. As a result, graph products can be defined in 256 distinct ways. The four graph products—the Cartesian, the Kronecker, the Strong, and the Lexicographic—are the standard graph products. Much work has been done on the distance matrix of these graph products, except for the Kronecker product. The Kronecker product of the graphs $\mathcal{G}$ and $\mathcal{H}$, denoted as $\mathcal{G}\otimes \mathcal{H}$, is a graph in which $(x_{i}, y_{j})\sim(x_{r},y_{s})$ if $x_{i} \sim x_{r}$ in $\mathcal{G}$ and $y_{j} \sim y_{s}$ in $\mathcal{H}$. 

The investigation of various types of spectra of product graphs is an intriguing topic. In \cite{indu}, Indulal and Gutman have determined all the distance spectra of certain product graphs. Atik and Panigrahi \cite{atik2} constructed the graphs to have a diameter greater than $d\in \{3,4, \cdots, 10\}$, which has exactly $d+1$ distinct distance eigenvalues.  In \cite{indu2}, Indulal has identified the distance eigenvalues of the Cartesian product of transmission regular graphs and the lexicographic product of graphs $\mathcal{G}$ and $\mathcal{H}$, where $\mathcal{H}$ is regular. Barik and Sahoo discuss the distance spectra of the corona operation of graphs in \cite{barik}. Recently, Atik et al. \cite{atik3} have identified all the distance eigenvalues of $m$-generation $n$-prism graphs.

A distance regular graph symbolized as $\mathcal{G}$ is characterized as being both regular and meeting the following criterion: for any pair of vertices $x$ and $y$ in $\mathcal{G}$ with a distance of $i$, there exists a constant number of neighbors $c_{i}$ and $b_{i}$ of $y$ at distances $i-1$ and $i+1$ from $x$, respectively. The numbers $c_{i}$ and $b_{i}$ are known as intersection numbers. Complete graphs, and cycle graphs are trivial examples of distance regular graphs. A Johnson graph $J(m, r)$ is defined as a graph with its vertex set comprising all possible $r$-subsets of an $m$-element set. Two vertices in $J(m, r)$ are considered adjacent if they intersect with exactly $r-1$ elements. This graph exhibits distance regularity with intersection numbers $c_{i}=i^2$, $b_{i}=(r-i)(m-r-i)$, and a diameter $d$ equal to the minimum of $r$ and $m-r$. As $J(m, r)$ is isomorphic to $J(m, m-r)$, we consider $m \ge 2r $ such that its diameter becomes $r.$ The distinct adjacency eigenvalues of $J(m,r)$, as obtained in \cite{brow}, are given by $\lambda_{i}=b_{i}-i$, each with a multiplicity of $\binom{n}{i}-\binom{n}{i-1}$ for $i=0,1, \cdots, r.$ The distinct distance eigenvalues of $J(m, r)$, as obtained in \cite{atik}, are given by $\mu_{0}=s, \mu_{1}=-\frac{s}{n-1}, \mu_{2}=0$ with multiplicity $1, (n-1), \binom{n}{m}-n$, respectively. The Hamming graph $H(d,q)$ of diameter $d$ has vertex set $q^d$, the set of ordered $d$-tuples of $q$-elements set. Two vertices are adjacent if they differ in precisely one coordinate. This is also a distance regular graph with intersection numbers $c_{i}=i, b_{i}=(d-i)(q-1).$  The distinct adjacency eigenvalues of $H(d,q)$, as obtained in \cite{brow}, are given by $\lambda_{i}=b_{0}-qc_{i}$ with multiplicity $\binom{d}{i}(q-1)^{i} $ for $i=0,1, \cdots, d.$ The distinct distance eigenvalues of $H(d,q)$, as obtained in \cite{indu2}, are given by $\mu_{0}=dq^{d-1}(q-1), \mu_{1}=-q^{(d-1)}, \mu_{2}=0$ with multiplicity $1, d(q-1), q^{d}-d(q-1)-1$ respectively.

Aalipour et al. \cite{Aali} determined the distance spectra of some graphs, including distance regular graphs, double odd graphs, and Doob graphs, and characterized strongly regular graphs as having more positive than negative distance eigenvalues. In \cite{atik}, Atik and Panigrahi found the distance spectrum of some distance regular graphs, including the well-known Johnson graphs.

As an operation of graphs, Kronecker product $\mathcal{G} \otimes \mathcal{H}$ was first introduced by Weichesel \cite{weichesel} in 1962. The Kronecker product has gained significant research attention in recent times as an effective technique for building larger networks. However, the distance spectra of the Kronecker product of graphs are not studied yet. In this article, we explore the complete distance spectrum of the Kronecker product applied to a specific set of families distance regular graphs in Theorem \ref{c2m} to Theorem \ref{hamdist}. For distance regular graphs, the distance matrix $D$ can be represented as a polynomial function of the adjacency matrix. We derive these polynomials for both the Johnson and Hamming graphs in Theorem \ref{johnpoly} and Theorem \ref{hampoly}. Additionally, we present two families of distance integral graphs, each with arbitrary diameters. Notably, these graph families are both adjacency integral and distance integral. These findings contribute to addressing the Open Problem presented in {\cite{indu3}}, which aims to identify graph families exhibiting both integral adjacency and distance spectra. 

Next, we have outlined a few useful existing results that will be used in the following sections.

A circulant matrix is characterized as a square matrix wherein each row, starting from the second row, features elements that are cyclically shifted one step to the right compared to the preceding row. It is denoted as 
 \begin{eqnarray*}\label{}\displaystyle {C}=circ(c_{0}, c_{1}, c_{2}, \cdots, c_{n-1})={\begin{bmatrix}  {c_{0}} &{c_{1}} &{c_{2}} &\cdots &{c_{n-2}} & {c_{n-1}}\\
 {c_{n-1}} &{c_{0}} &{c_{1}} &\ddots &{c_{n-3}}& {c_{n-2}}\\
{c_{n-2}} &{c_{n-1}} &{c_{0}} &\ddots &{c_{n-4}}& {c_{n-3}}\\
 \ddots &\ddots &\ddots &\ddots &\ddots &\ddots  \\
{c_{2}}&{c_{3}}&{c_{4}}& \ddots &{c_{0}}&{c_{1}}\\
{c_{1}} &{c_{2}} &{c_{3}} &\ddots &{c_{n-1}}& {c_{0}}
 \end{bmatrix}}.
 \end{eqnarray*}
The eigenvectors and eigenvalues of matrix $ C $ can be represented using the elements of $ C $ and the $n$ th root of unity. We have written this well-known result in the following form:
\begin{lemma}\label{adjc}
Let ${C}=circ(c_{0}, c_{1}, c_{2}, \cdots, c_{n-1})$ and $\rho_{j}=e^{i\pi j/n}=\cos(\frac{2\pi}{n}j) + i\sin (\frac{2\pi}{n}j)$ for $j=0, 1, 2, \cdots, n-1$. Then $V^{(j)}$ is the eigenvector corresponding to the eigenvalues $\lambda_{j}$, where
 \begin{eqnarray*}V^{(j)}=\displaystyle \begin{bmatrix}
    {1}\\
    {\rho_{j}}\\
    {\rho_{j}^{2}}\\
    \vdots\\
    {\rho_{j}^{n-1}}\\
\end{bmatrix} \mbox{~~and~~} \lambda_{j}= c_{0} + c_{1}\rho_{j} + c_{2}\rho_{j}^2 + c_{3}\rho_{j}^3+ \cdots + c_{n-1}\rho_{j}^{n-1}.\end{eqnarray*}
\end{lemma}
Using the above eigenvalue and eigenvector relation on the circulant matrix, we can write the following lemma:

\begin{lemma}
    If $\lambda_{j}= a_{0} + a_{1}\rho_{j} + a_{2}\rho_{j}^2 + a_{3}\rho_{j}^3+ \cdots + a_{n-1}\rho_{j}^{n-1}$ is an eigenvalue of $A=circ(a_{0}, a_{1}, a_{2}, \cdots, a_{n-1})$ and $\mu_{j}= b_{0} + b_{1}\rho_{j} + b_{2}\rho_{j}^2 + b_{3}\rho_{j}^3+ \cdots + b_{n-1}\rho_{j}^{n-1}$ is an eigenvalue of $B=circ(b_{0}, b_{1}, b_{2}, \cdots, b_{n-1})$, then for any two scalars $s$ and $t$, the eigenvalues of $sA+tB$, are $s\lambda_{j}+t\mu_{j}$ for $j=0, 1, 2, \cdots, n-1.$  
\end{lemma}
Next, we discuss the eigenvalues of the block circulant matrices. The class of block circulant matrices is denoted as $\mathcal{BC}_{n,k}$, with matrix partitioned by $n$ and all square submatrices of order $k>1$. Let ${B}\in \mathcal{BC}_{n,k}$, then ${B}=Circ(\textbf{b}_{0}, \textbf{b}_{1}, \textbf{b}_{2}, \cdots, \textbf{b}_{n-1}) $ where $\textbf{b}_{i}$'s are square submatrices of order $k$ for $i=0,1,2, \cdots, (n-1).$ 
\begin{eqnarray*}\label{}\displaystyle {B}=Circ(\textbf{b}_{0}, \textbf{b}_{1}, \textbf{b}_{2}, \cdots, \textbf{b}_{n-1})={\begin{bmatrix}  \textbf{b}_{0} &\textbf{b}_{1} &\textbf{b}_{2} &\cdots &\textbf{b}_{n-2} & \textbf{b}_{n-1}\\
 \textbf{b}_{n-1} &\textbf{b}_{0} &\textbf{b}_{1} &\ddots &\textbf{b}_{n-3}& \textbf{b}_{n-2}\\
\textbf{b}_{n-2} &\textbf{b}_{n-1} &\textbf{b}_{0} &\ddots &\textbf{b}_{n-4}& \textbf{b}_{n-3}\\
\ddots &\ddots &\ddots &\ddots &\ddots &\ddots  \\
\textbf{b}_{2}&\textbf{b}_{3}&\textbf{b}_{4}& \ddots &\textbf{b}_{0}&\textbf{b}_{1}\\
\textbf{b}_{1} &\textbf{b}_{2} &\textbf{b}_{3} &\ddots &\textbf{b}_{n-1}& \textbf{b}_{0}
 \end{bmatrix}}.
 \end{eqnarray*}
 
In \cite{garry}, the eigenvalues are computed for the block circulant matrix with the help of smaller size matrices $H_{j}$. We have reduced the result for the real symmetric block circulant matrix $B$ in theorem form as follows:
\begin{theorem}\emph{\cite{garry}} \label{bcir}
  Let  ${B}=Circ(\textbf{b}_{0}, \textbf{b}_{1}, \textbf{b}_{2}, \cdots, \textbf{b}_{n-1}) $ be a real symmetric block circulant matrix, where $\textbf{b}_{i}$'s are square submatrices of order $k$ for $i=0,1,2, \cdots, (n-1).$ Let 
  \begin{eqnarray*}
  {H}_{j}= \textbf{b}_{0}+ \displaystyle\sum_{f=1}^{h-1}\big[\textbf{b}_{f}\rho_{j}^{f}+\textbf{b}_{f}^T\Bar{\rho_{j}^{f}}\big] +\bigg\{\begin{array}{cc}
          \textbf{0}~~~~~~& \mbox{if $n=2h-1$} \\
         \textbf{b}_{h}(-1)^{j} ~~~~~~& \mbox{if $n=2h$}.
         \end{array}
         \end{eqnarray*} Then all the eigenvalues of ${B}$ are the union of the eigenvalues of all ${H}_{j},$ $j=0, 1, 2, \cdots, n-1.$
\end{theorem}
The following results provide information regarding the connectivity and diameter of the graph resulting from the Kronecker product and the Cartesian product.
\begin{lemma}\emph{\cite{weichesel}} \label{connected}
Let $\mathcal{G}=\mathcal{G}_{1}\otimes \mathcal{G}_{2}$ be the Kronecker product of simple connected graphs $\mathcal{G}_{1}$ and $\mathcal{G}_{2}$. Then $\mathcal{G}$ is connected if and only if either $\mathcal{G}_{1}$ or $\mathcal{G}_{2}$  contains an odd cycle.  
\end{lemma}
\begin{lemma} \emph{(\cite{imri}, Lemma 1.2)} \label{cerconnec}
    A Cartesian product $\mathcal{G} \Box \mathcal{H}$ is connected if and only if both factors are connected.
\end{lemma}
Let $\gamma (\mathcal{G}; x,y)$ denote the minimum integer such that there exists an $(x,y)$-walk of length $k$ for any $k\ge \gamma (\mathcal{G}; x,y). $ And $\gamma (\mathcal{G})$  is defined by,
$\gamma (\mathcal{G})= max\{ \gamma(\mathcal{G}; x,y): x, y \in V(\mathcal{G})\}$. Now the next result tells us about the diameter of the Kronecker product of graphs.

Let $\gamma (\mathcal{G}; x, y)$ represent the smallest integer ensuring the existence of a walk of length $k$ for any $k\ge \gamma (\mathcal{G}; x, y)$ between vertices $x$ and $y$. The quantity $\gamma (\mathcal{G})$ is defined as follows: \[\gamma (\mathcal{G}) = \max\{ \gamma(\mathcal{G}; x, y): x, y \in V(\mathcal{G})\}.\] Now the next result tells us about the diameter of the Kronecker product of graphs.

\begin{theorem}\emph{\cite{hu}} \label{dia}
    Let $\mathcal{G}$ be a connected graph with diameter $d\ge 1$ and $\mathcal{H}$ be a complete $t$-partite graph with $t>3.$ Then 
    \begin{eqnarray*}
        d(\mathcal{G} \otimes \mathcal{H})=\left \{\begin{array}{cc}
           d,  & d\ge3;  \\
          2,   &  d \leq 2 ~~and~~ \gamma(G) \leq 2;\\
          3, & d \leq 2 ~~and~~ \gamma(G) \ge 2
\end{array}. \right.
    \end{eqnarray*}
\end{theorem}

    A square Vandermode matrix is an $(n+1)\times (n+1)$ matrix of the form
    \begin{eqnarray*}\label{}\displaystyle {V}=V({x}_{0}, {x}_{1}, {x}_{2}, \cdots, {x}_{n})={\begin{bmatrix}  {1} & {x}_{0} & {x}_{0}^{2} &\cdots &{x}_{0}^{n} \\
 {1} & {x}_{1} & {x}_{1}^{2} &\cdots & {x}_{1}^{n} \\
 {1} & {x}_{2} & {x}_{2}^{2} &\cdots & {x}_{2}^{n} \\
\vdots &\vdots &\vdots &\ddots &\vdots   \\
{1} & {x}_{n}& {x}_{n}^{2}& \cdots & {x}_{n}^{n}
 \end{bmatrix}}.
 \end{eqnarray*}
 
The determinant of the above Vandermonde matrix, $ det(V)=\displaystyle \prod_{0\le i < j \le n} (x_{j}-x_{i})$, is called Vandermonde determinant or Vandermonde polynomial. 
The Vandermonde determinant is nonzero exclusively when all $x_{i}$ values are distinct. So, if all $x_{i}$ are distinct, then the square Vandermonde matrix is invertible. The inverse matrix $V^{-1}$ can be computed by Lagrange interpolation.
\begin{eqnarray}\label{vaninv}V^{-1}=\displaystyle {\begin{bmatrix}  {1} & {x}_{0} & {x}_{0}^{2} &\cdots &{x}_{0}^{n} \\
 {1} & {x}_{1} & {x}_{1}^{2} &\cdots & {x}_{1}^{n} \\
 {1} & {x}_{2} & {x}_{2}^{2} &\cdots & {x}_{2}^{n} \\
\vdots &\vdots &\vdots &\ddots &\vdots   \\
{1} & {x}_{n}& {x}_{n}^{2}& \cdots & {x}_{n}^{n}
 \end{bmatrix}^{-1}}=L=\displaystyle {\begin{bmatrix}  {L}_{00} & {L}_{01} &\cdots &{L}_{0n} \\
  {L}_{10} & {L}_{11} &\cdots & {L}_{1n} \\
  {L}_{20} & {L}_{21} &\cdots & {L}_{2n} \\
\vdots &\vdots &\ddots &\vdots   \\
 {L}_{n0}& {L}_{n1}& \cdots & {L}_{nn}
 \end{bmatrix}}.
 \end{eqnarray}
 The columns of the inverse matrix are the coefficients of the Lagrange polynomials $L_{j}(x)=L_{0j}+L_{1j}x+ \cdots +L_{nj}x^{n}=\displaystyle \prod_{\substack{0\le i \le n\\ i\neq j}}\frac{(x-x_{i})}{(x_{j}-x_{i})}.$ This is easily demonstrated: the polynomials clearly satisfy $L_{j}(x_{i})=0$ for $i\neq j$ while $L_{j}(x_{j})=1.$ So, we may compute the product $VL=[L_{j}(x_{i})]_{i,j=0}^{n}=I$, the identity matrix.

\section{Some preliminary results}
In this section, we have derived some results in a specialized format that will be utilized in the forthcoming sections.

Let $\mathbb{C}$ be the set of all complex numbers. Here we prove an identity which is the sum of a finite series whose terms are term by term product of some A. P. and G. P. series.
\begin{identity} \label{ide}
Let $ a, d, r \in \mathbb{C}$ and $r\neq 1 $. Then ~ $a + (a+d)r + (a+2d)r^2 + (a+3d)r^3 + \cdots + (a+(n-1)d)r^{n-1} = \left[a+(n-1)d\right]\frac{r^n -1}{r-1} - \frac{d}{r-1}\left[\frac{r^n -1}{r-1} - n\right]$
\end{identity}

\begin{proof} We can write the following steps as follows:
\begin{eqnarray}\displaystyle
   \nonumber a + ar + ar^2 + ar^3 + \cdots + ar^{n-1} &=& a\frac{r^n -1}{r -1}\\
\nonumber dr + dr^2 + dr^3 + \cdots + dr^{n-1} &=& d\frac{r^n -1}{r -1} - d\frac{(r -1)}{r -1}\\
 \nonumber dr^2 + dr^3 + \cdots + dr^{n-1} &=& d\frac{r^n -1}{r -1} - d\frac{(r^2 -1)}{r -1}\\
  \nonumber \ddots ~~~~ \ddots ~~~~\ddots && ~~~~~~~ \vdots  ~~~ ~~~~~~~~~\vdots   \\
 \nonumber dr^{n-2} + dr^{n-1} &=& d\frac{r^n -1}{r -1} - d\frac{(r^{n-2} -1)}{r -1}\\
 \nonumber dr^{n-1} &=& d\frac{r^n -1}{r -1} - d\frac{(r^{n-1} -1)}{r -1}
\end{eqnarray}
Adding all the above equations
\begin{eqnarray*}
 a + (a+d)r &+& (a+2d)r^2 + (a+3d)r^3  + \cdots+ [a+(n-1)d]r^{n-1}\\
 &=&[a+(n-1)d]\frac{r^n -1}{r-1} - \frac{d}{r-1}[r+r^2+r^3+ \cdots + r^{n-1} - (n-1)]\\
 &=&[a+(n-1)d]\frac{r^n -1}{r-1} - \frac{d}{r-1} \bigg[ \frac{r^n -1}{r-1} - n \bigg].
\end{eqnarray*}
\end{proof}
The adjacency matrix for the cyclic graph $C_n$ can be expressed as a circulant matrix, denoted as $A = circ(0, 1, 0, 0, \ldots, 0, 1)$. Then by using Lemma \ref{adjc}, the eigenvectors and eigenvalues of $A$ are 
${V}^{(j)}=\displaystyle \begin{bmatrix}
    {1}\\
    {\rho_{j}}\\
    {\rho_{j}^{2}}\\
    \vdots\\
    {\rho_{j}^{n-1}}\\
\end{bmatrix}$,
~$\lambda_{j}=\rho_{j}+\rho_{j}^{n-1}=\rho_{j}+\overline{\rho_{j}}= 2\cos\left(\frac{2\pi}{n}j\right);$ for $j=0, 1, 2, \cdots, n-1.$

 In \cite{grao}, the distance eigenvalues of the Cycle graph ${C}_{n}$ are already discussed. In this context, we determine the eigenvalues of the linear combination formed by combining the adjacency matrix and distance matrix of the cyclic graph ${C}_{n}$. Here we use the notation $\mathcal{\Re}[z]=$ real part of $z$.
\begin{lemma}\label{lam2}
    Consider matrices $A$ and $D$ as the adjacency matrix and distance matrix of the Cycle graph ${C}_{n}$, respectively, with identical vertex ordering. Then for any two scalars $s$ and $t$, the eigenvalues of $sA+tD$ are
\vspace{-.32cm}
\begin{center}
\begin{table}[!ht]
\begin{tabular}{|cc|cc|cc|}
    \hline
   $n$ & \multicolumn{1}{|c|}{ $j=0$}&\multicolumn{1}{|c|}{ $j~ even$} & \multicolumn{1}{|c|}{ $j~ odd$}\\
    \hline
    $even$ & \multicolumn{1}{|c|}{$2s+\frac{n^2}{4}t$ }&\multicolumn{1}{|c|}{$2s\cos\left(\frac{2\pi}{n}j\right)$ } & \multicolumn{1}{|c|}{$2s\cos\left(\frac{2\pi}{n}j\right)-t \mathrm{cosec}^{2}\left(\frac{\pi}{n}j\right)$ } \\
   \hline
   $odd$ & \multicolumn{1}{|c|}{$2s+\frac{n^2-1}{4}t$ }&\multicolumn{1}{|c|}{$2s\cos\left(\frac{2\pi}{n}j\right) -\frac{t}{4}\sec^2\left(\frac{\pi}{2n}j\right)$ } & \multicolumn{1}{|c|}{$2s\cos\left(\frac{2\pi}{n}j\right) -\frac{t}{4}\mathrm{cosec}^{2}\left(\frac{\pi}{2n}j\right)$} \\
\hline
\end{tabular}
\end{table}
\end{center}
\end{lemma}
\begin{proof}
  Case I: For $n$ is odd, the distance matrix of ${C}_n$ is $D=circ \left(0, 1, 2, \cdots, \frac{n-1}{2}, \frac{n-1}{2}, \cdots, 2, 1\right).$\\
The eigenvalues of $D$ corresponding the eigenvector ${V}^{(j)}$, mentioned earlier are 
\begin{eqnarray*}
\mu_{j}&=& 0 + \rho_{j} + 2\rho_{j}^2 + \cdots +\left(\frac{n-1}{2}\right)\rho_{j}^{\frac{n-1}{2}} + \left(\frac{n-1}{2}\right)\rho_{j}^{\frac{n+1}{2}} + \cdots + 2\rho_{j}^{n-2} + \rho_{j}^{n-1};\\
 &&~~~~~~~~~~~~~~~~~~~~~~~~~~~~~~~~~~~~~~~~~~~~~~~~~~~~~~~~~~~~~~~~~~~~ \mbox{for} ~~j=0, 1, 2, \cdots, (n-1).\\
\mbox{or,}~~ \mu_{j}&=& \left(\rho_{j} + \rho_{j}^{n-1}\right) + 2\left(\rho_{j}^2 + \rho_{j}^{n-2}\right)+ \cdots +\left(\frac{n-1}{2}\right)\left(\rho_{j}^{\frac{n-1}{2}} + \rho_{j}^{\frac{n+1}{2}} \right)\\
&=&\left(\rho_{j} + \overline{{\rho_{j}}}\right) + 2\left(\rho_{j}^2 + 
 \overline{{\rho_{j}^{2}}}\right)+ \cdots +\left(\frac{n-1}{2}\right)\left(\rho_{j}^{\frac{n-1}{2}} + \overline{{\rho_{j}^{\frac{n-1}{2}}}} \right)\\
&=& 2 \mathcal{\Re} \left[\rho_{j} + 2\rho_{j}^2 + \cdots +\left(\frac{n-1}{2}\right)\rho_{j}^{\frac{n-1}{2}}\right]\\
\mbox{So,}~~ \mu_{0}&=& 2\left[1+2+ \cdots +\frac{n-1}{2}\right]= \frac{n^2 -1}{4}
\end{eqnarray*}
and using the Identity \ref{ide}, we get
\begin{eqnarray*}
\mu_{j}&=& 2\mathcal{\Re} \left[ \rho_{j} \left\{ \left(\frac{n-1}{2}\right)\frac{(\rho_{j}^{\frac{n-1}{2}}-1)}{\rho_{j}-1}- \frac{1}{\rho_{j}-1}\left(\frac{\rho_{j}^{\frac{n-1}{2}}-1}{\rho_{j}-1} - \frac{n-1}{2}\right)\right\}\right];~ \mbox{for}~ j=1,2, \cdots, (n-1).\\
&=& 2\Re \left[ \left(\frac{n-1}{2}\right) \rho_{j}^{\frac{n+1}{4}} \frac{(\rho_{j}^{\frac{n-1}{4}}-\rho_{j}^{-\frac{n-1}{4}})}{(\rho_{j}^{\frac{1}{2}}-\rho_{j}^{-\frac{1}{2}})} - \rho_{j}^{\frac{n-1}{4}} \frac{(\rho_{j}^{\frac{n-1}{4}}-\rho_{j}^{-\frac{n-1}{4}})}{(\rho_{j}^{\frac{1}{2}}-\rho_{j}^{-\frac{1}{2}})^2} + \left(\frac{n-1}{2}\right)\frac{\rho_{j}^{\frac{1}{2}}}{(\rho_{j}^{\frac{1}{2}}-\rho_{j}^{-\frac{1}{2}})} \right]\\
&=& 2\Re \left[ \left(\frac{n-1}{2}\right) \rho_{j}^{\frac{n+1}{4}} \frac{\sin((\frac{n-1}{4})\frac{2\pi}{n}j)}{\sin(\frac{\pi}{n}j)} - \rho_{j}^{\frac{n-1}{4}} \frac{\sin((\frac{n-1}{4})\frac{2\pi}{n}j)}{2i\sin^2(\frac{\pi}{n}j)} + \left(\frac{n-1}{2}\right)\frac{\rho_{j}^{\frac{1}{2}}}{2i \sin(\frac{\pi}{n}j)} \right] \\
&=& 2 \left[ \left(\frac{n-1}{2}\right) \frac{\cos((\frac{n-1}{4})\frac{2\pi}{n}j) \sin((\frac{n-1}{4})\frac{2\pi}{n}j)}{\sin(\frac{\pi}{n}j)} - \frac{\sin^2((\frac{n-1}{4})\frac{2\pi}{n}j)}{2\sin^2(\frac{\pi}{n}j)} + \left(\frac{n-1}{2}\right)\frac{\sin(\frac{\pi}{n}j)}{2 \sin(\frac{\pi}{n}j)} \right]\\
&=&  \left(\frac{n-1}{2}\right) \frac{\sin(\pi j) - \sin(\frac{\pi}{n}j)}{\sin(\frac{\pi}{n}j)} - \frac{\sin^2(\frac{\pi}{2}j -\frac{\pi}{2n}j)}{\sin^2(\frac{\pi}{n}j)} + \left(\frac{n-1}{2}\right)\\
&=&-\left(\frac{n-1}{2}\right)  - \frac{\sin^2(\frac{\pi}{2}j -\frac{\pi}{2n}j)}{\sin^2(\frac{\pi}{n}j)} + \left(\frac{n-1}{2}\right)\\
&=&  - \frac{\sin^2(\frac{\pi}{2}j -\frac{\pi}{2n}j)}{\sin^2(\frac{\pi}{n}j)}\\
  &=&\left\{ \begin{array}{rr}
         -\frac{\sin^2\left(\frac{\pi}{2n}j\right)}{\sin^2\left(\frac{\pi}{n}j\right)};~&\mbox{ for $j$ is even}\\
         -\frac{\cos^2\left(\frac{\pi}{2n}j\right)}{\sin^2\left(\frac{\pi}{n}j\right)};~&\mbox{ for $j$ is odd}
       \end{array} \right.\\
&=&\left\{ \begin{array}{rr}
         -\frac{1}{4}\sec^2\left(\frac{\pi}{2n}j\right); ~&\mbox{ for $j$ is even}\\
 -\frac{1}{4} \mathrm{cosec}^{2}\left(\frac{\pi}{2n}j\right); ~&\mbox{ for $j$ is odd}
\end{array} \right.
\end{eqnarray*}
Case II: If $n$ is even, the distance matrix of $C_n$ is $D=circ\left(0, 1, 2, \cdots, \frac{n-2}{2}, \frac{n}{2}, \frac{n-2}{2}, \cdots, 2, 1\right).$\\
The eigenvalues of $D$ corresponding to the eigenvector ${V}^{(j)}$, mentioned earlier, are 
\begin{eqnarray*}
    \mu_{j}&=& 0 + \rho_{j} + 2\rho_{j}^2 + \cdots +\left(\frac{n-2}{2}\right)\rho_{j}^{\frac{n}{2}-1} + \frac{n}{2}\rho_{j}^{\frac{n}{2}}+ \left(\frac{n-2}{2}\right)\rho_{j}^{\frac{n}{2}+1} + \cdots + 2\rho_{j}^{n-2} + \rho_{j}^{n-1};\\ &&~~~~~~~~~~~~~~~~~~~~~~~~~~~~~~~~~~~~~~~~~~~~~~~~~~~~~~~~~~~~~~~~~~~~~~~~~~~~~~\mbox{for} ~ j=0, 1, 2, \cdots, (n-1).\\
\mbox{or,}~~ \mu_{j}&=& \left(\rho_{j} + \rho_{j}^{n-1}\right) + 2\left(\rho_{j}^2 + \rho_{j}^{n-2}\right) + \cdots +\left(\frac{n-2}{2}\right)\left(\rho_{j}^{\frac{n}{2}-1} + \rho_{j}^{\frac{n}{2}+1} \right) + \frac{n}{2}\rho_{j}^{\frac{n}{2}}\\
&=&\left(\rho_{j} + \overline{{\rho_{j}}}\right) + 2\left(\rho_{j}^2 + \overline{{\rho_{j}^{2}}}\right)+ \cdots +\left(\frac{n-2}{2}\right)\left(\rho_{j}^{\frac{n}{2}-1} + \overline{{\rho_{j}^{\frac{n}{2}-1}}} \right) + \frac{n}{2}\rho_{j}^{\frac{n}{2}}\\
 &=&2 \Re \left[\rho_{j} + 2\rho_{j}^2 + \cdots +\left(\frac{n-2}{2}\right)\rho_{j}^{\frac{n}{2}-1} \right] +\frac{n}{2}\cos(\pi j)\\
&=&2\Re \left[\rho_{j} \left(1 + 2\rho_{j} + 3 \rho_{j}^2 + \cdots +\left(\frac{n-2}{2}\right)\rho_{j}^{\frac{n}{2}-2}\right) \right] +\frac{n}{2}\cos(\pi j)\\
\mbox{So,}~~~ \mu_{0}&=& 2\left[1+2+ \cdots +\frac{n-2}{2}\right] + \frac{n}{2}= \frac{n^2}{4}
\end{eqnarray*}
and using the Identity \ref{ide}, we get
\begin{eqnarray*}
    \mu_{j}&=& 2\Re \left[ \rho_{j} \left\{ \left(\frac{n-2}{2}\right)\frac{(\rho_{j}^{\frac{n-2}{2}}-1)}{\rho_{j}-1}- \frac{1}{\rho_{j}-1}\left(\frac{\rho_{j}^{\frac{n-2}{2}}-1}{\rho_{j}-1} - \frac{n-2}{2}\right)\right\}\right] + \frac{n}{2}\cos(\pi j) ; \\
    &&~~~~~~~~~~~~~~~~~~~~~~~~~~~~~~~~~~~~~~~~~~~~~~~~~~~~~~~~~~~~~~~~~~~~~~~~~~~~~~~~~~~\mbox{for} ~~j=1,2, \cdots, (n-1).\\
&=&2\Re \left[ \left(\frac{n-2}{2}\right) \rho_{j}^{\frac{n}{4}} \frac{(\rho_{j}^{\frac{n-2}{4}}-\rho_{j}^{-\frac{n-2}{4}})}{(\rho_{j}^{\frac{1}{2}}-\rho_{j}^{-\frac{1}{2}})} - \rho_{j}^{\frac{n-2}{4}} \frac{(\rho_{j}^{\frac{n-2}{4}}-\rho_{j}^{-\frac{n-2}{4}})}{(\rho_{j}^{\frac{1}{2}}-\rho_{j}^{-\frac{1}{2}})^2} + \left(\frac{n-2}{2}\right)\frac{\rho_{j}^{\frac{1}{2}}}{(\rho_{j}^{\frac{1}{2}}-\rho_{j}^{-\frac{1}{2}})} \right] + \frac{n}{2}\cos(\pi j)\\
&=& 2\Re \left[ \left(\frac{n-2}{2}\right) \rho_{j}^{\frac{n}{4}} \frac{\sin((\frac{n-2}{4})\frac{2\pi}{n}j)}{\sin(\frac{\pi}{n}j)} - \rho_{j}^{\frac{n-2}{4}} \frac{\sin((\frac{n-2}{4})\frac{2\pi}{n}j)}{2i\sin^2(\frac{\pi}{n}j)} + \left(\frac{n-2}{2}\right)\frac{\rho_{j}^{\frac{1}{2}}}{2i \sin(\frac{\pi}{n}j)} \right] + \frac{n}{2}\cos(\pi j) \\
&=& 2 \left[ \left(\frac{n-2}{2}\right) \frac{\cos(\frac{\pi}{2}j) \sin((\frac{n-2}{4})\frac{2\pi}{n}j)}{\sin(\frac{\pi}{n}j)} - \frac{\sin^2((\frac{n-2}{4})\frac{2\pi}{n}j)}{2\sin^2(\frac{\pi}{n}j)} + \left(\frac{n-2}{2}\right)\frac{\sin(\frac{\pi}{n}j)}{2 \sin(\frac{\pi}{n}j)} \right] + \frac{n}{2}\cos(\pi j)\\
&=& \left[ \left(\frac{n-2}{2}\right)\frac{\sin(\pi j - \frac{\pi}{n}j) - \sin(\frac{\pi}{n}j)}{\sin(\frac{\pi}{n}j)}- \frac{\sin^2(\frac{\pi}{2}j -\frac{\pi}{n}j)}{2\sin^2(\frac{\pi}{n}j)} + \left(\frac{n-2}{2}\right) \right] + \frac{n}{2} \cos(\pi j)   \\
&=&\left\{ \begin{array}{rr}
        -2\left(\frac{n-2}{2}\right) -1 +\frac{n-2}{2} + \frac{n}{2};~~~~&\mbox{ for $j$ is even}\\
         -\frac{\cos^2\left(\frac{\pi}{n}j\right)}{\sin^2\left(\frac{\pi}{n}j\right)} +\frac{n-2}{2} - \frac{n}{2};~~~&\mbox{ for $j$ is odd}
       \end{array} \right.\\
&=&\left\{ \begin{array}{rr}
        0;~~~&\mbox{ for $j$ is even}\\
         -\mathrm{cosec}^{2}\left(\frac{\pi}{n}j\right);~~~&\mbox{ for $j$ is odd}
       \end{array} \right.
\end{eqnarray*}
Thus, we have eigenvalues of the adjacency matrix $A$ are $\lambda_{j}= 2\cos\left(\frac{2\pi}{n}j\right)$ and eigenvalues of the distance matrix $D$ are $\mu_{j}$ where as in both the cases corresponding eigenvectors are ${V}^{(j)}=\displaystyle \begin{bmatrix}
    {1}\\
    {\rho_{j}}\\
    {\rho_{j}^{2}}\\
    \vdots\\
    {\rho_{j}^{n-1}}
\end{bmatrix},$ for $j=0, 1, 2, \cdots, n-1$.
Hence, one can write the spectrum of a linear combination of the adjacency matrix and distance matrix of Cycle graph ${C}_{n}$ as mentioned in the statement of this lemma.
\end{proof}


\section{Distance matrix as a polynomial of adjacency matrix}
In this section, we discuss the distance matrix of a distance regular graph with diameter $d$, which can be written as a polynomial of the adjacency matrix of degree $d$. However, determining such a polynomial is not an easy task. Nonetheless, we have identified such polynomials for two distinct families of distance regular graphs—the Johnson graph and the Hamming graph. While we have found these polynomials in terms of both adjacency spectra and distance spectra, we can also convert them into the intersection numbers of the graph.

\begin{theorem}
    \label{johnpoly}
Let J(n,m) be the Johnson graph $(n \ge 2m)$ with intersection numbers $c_{i}$ and $b_{i}$. Let $A$ and $D$ be the adjacency matrix and distance matrix. Then the polynomial $p(x)=s\left\{\displaystyle \prod_{i=1}^m\frac{(x-b_{i}+i)}{(b_{0}-b_{i}+i)}- \frac{1}{n-1} \displaystyle \prod_{\substack{i=0\\ i\neq1}}^m \frac{(x-b_{i}+i)}{(b_{1}-b_{i}+i-1)}\right\}$ is the required polynomial satisfy $D=p(A)$, where $s=\displaystyle \sum_{j=0}^{m} jk_{j}$ and $k_{j}=\binom{m}{j}\binom{n-m}{j}$ for $j=0,1, \cdots, m.$ 
\end{theorem}
\begin{proof}
 We know that the intersection numbers of the Johnson graph $J(n,m)$ are $c_{i}=i^2, b_{i}=(m-i)(n-m-i).$ The distinct eigenvalues of $A$ are $\lambda_{i}=b_{i}-i$ with multiplicity $\binom{n}{i}-\binom{n}{i-1} $ for $i=0,1, \cdots, m.$ The distinct eigenvalues of $D$ are $\mu_{0}=s, \mu_{1}=-\frac{s}{n-1}, \mu_{2}=0$ with multiplicity $1, (n-1), \binom{n}{m}-n$ respectively. For every distance regular graph with diameter $r,$ there exists a polynomial $p(x)$ of degree $r$ such that $p(A)=D.$ So, $p(\lambda_{i}) \in \{ \mu_{0}, \mu_{1},\mu_{2} \}.$ \textit{Claim:} $p(\lambda_{0})=\mu_{0}, p(\lambda_{1})=\mu_{1}$ and $p(\lambda_{i})=\mu_{2}=0$ for $i=2,3, \cdots, m.$ It is obvious that $p(\lambda_{0})=\mu_{0},$ because $p(\lambda_{0})$ is an eigenvalue of the matrix $p(A)$ and $\mu_{0}$ is the eigenvalue of the matrix $D$ corresponding the same eigenvector $j_{\binom{n}{m}}.$ So, mainly we have to show $p(\lambda_{1})=\mu_{1}.$ We know, if $\lambda$ has multiplicity $\alpha$ then $p(\lambda)$ has multiplicity at least $\alpha.$ To prove $p(\lambda_{1})=\mu_{1},$ we prove that $p(\lambda_{i})\neq \mu_{1}$ for $i=2, 3, \cdots, m.$
 For this we have to show that all $\lambda_{i},$ for $i=2, 3, \cdots, m;$ have more multiplicity than $\lambda_{1}.$
 
 Now, \begin{eqnarray*}
     \binom{n}{i}-\binom{n}{i-1}&=&\frac{n!}{(n-i)!i!}-\frac{n!}{(n-i+1)!(i-1)!}\\
     &=& \frac{n!}{(n-i+1)!(i-1)!}\left[\frac{n-i+1}{i}-1\right]\\
     &=& \binom{n}{i-1}\left[\frac{n+1}{i}-2\right]\\
     &>& \binom{n}{i-1} \frac{2}{n} ~~~~~~~\left[\because \left(\frac{n+1}{i}-2\right)>\frac{2}{n} ~\mbox{for}~ i=2, 3, \cdots, m; ~\mbox{and}~  2m \leq n\right]\\
     &=& \frac{2}{n} \times \frac{n}{i-1}\binom{n-1}{i-2} ~~~~\left[\because r\binom{n}{r}=n\binom{n-1}{r-1}\right]\\
     &=&\frac{2}{i-1}\binom{n-1}{i-2}\\
     &>&\frac{2}{i-1}\binom{n-1}{2}   ~~~~~~~~\left[\mbox{for}~ 4\leq i \leq \frac{n}{2}\right]\\
     &=& \frac{(n-1)(n-2)}{i-1} 
     > (n-1) ~~~~~~\left[\because \frac{n-2}{i-1}>1 ~ \mbox{for}~ n\geq 8,~ i\geq 4 \right].
 \end{eqnarray*}
 For $i=2, n\geq 6; \binom{n}{2}-\binom{n}{1} = \frac{n(n-3)}{2}>n>(n-1).$\\
 For $i=3, n\geq 7; \binom{n}{3}-\binom{n}{2} = \frac{n(n-1)(n-5)}{6}>(n-1).$\\
 So, for $i=2, 3, \cdots, m$ and $n\geq 8;$ each $p(\lambda_{i})$ has multiplicity more than $(n-1).$ This implies $p(\lambda_{i})\neq \mu_{1}$ for $i=2, 3, \cdots, m.$ Hence $p(\lambda_{0})=\mu_{0}, p(\lambda_{1})=\mu_{1}$ and $p(\lambda_{i})=\mu_{2}=0$ for $i=2,3, \cdots, m.$
 Let $p(x)=a_{0} + a_{1}x + a_{2}x^{2} + \cdots + a_{m}x^{m}$ be polynomial of degree $m$ which satisfies $p(\lambda_{0})=\mu_{0}, p(\lambda_{1})=\mu_{1}, p(\lambda_{i})=\mu_{2}=0$ for $i=2, 3, \cdots, m.$ This problem can be viewed in matrix form as follows:
 \begin{eqnarray*}\label{}\displaystyle {\begin{bmatrix}  {1} & {\lambda}_{0} & {\lambda}_{0}^{2} &\cdots &{\lambda}_{0}^{m} \\
 {1} & {\lambda}_{1} & {\lambda}_{1}^{2} &\cdots & {\lambda}_{1}^{m} \\
 {1} & {\lambda}_{2} & {\lambda}_{2}^{2} &\cdots & {\lambda}_{2}^{m} \\
\vdots &\vdots &\vdots &\ddots &\vdots   \\
{1} & {\lambda}_{m}& {\lambda}_{m}^{2}& \cdots & {\lambda}_{m}^{m}
 \end{bmatrix}}
 \displaystyle {\begin{bmatrix}   {a}_{0}  \\
  {a}_{1}  \\
   {a}_{2} \\
\vdots    \\
 {a}_{m}
 \end{bmatrix}}=\displaystyle {\begin{bmatrix}   p({\lambda}_{0})  \\
  p({\lambda}_{1})  \\
   p({\lambda}_{2}) \\
\vdots    \\
 p({\lambda}_{m})
 \end{bmatrix}}=\displaystyle {\begin{bmatrix}   {\mu}_{0}  \\
  {\mu}_{1}  \\
   {0} \\
\vdots    \\
 {0}
 \end{bmatrix}}
 \end{eqnarray*}
 Using equation \ref{vaninv} we get,
 \begin{eqnarray*}
 \displaystyle {\begin{bmatrix}   {a}_{0}  \\
  {a}_{1}  \\
   {a}_{2} \\
\vdots    \\
 {a}_{m}
 \end{bmatrix}}=\displaystyle {\begin{bmatrix}  {L}_{00} & {L}_{01} &\cdots &{L}_{0m} \\
  {L}_{10} & {L}_{11} &\cdots & {L}_{1m} \\
  {L}_{20} & {L}_{21} &\cdots & {L}_{2m} \\
\vdots &\vdots &\ddots &\vdots   \\
 {L}_{m0}& {L}_{m1}& \cdots & {L}_{mm}
 \end{bmatrix}} \displaystyle {\begin{bmatrix}   {\mu}_{0}  \\
  {\mu}_{1}  \\
   {0} \\
\vdots    \\
 {0}
 \end{bmatrix}}=\displaystyle {\begin{bmatrix}   {\mu}_{0}{L}_{00}+{\mu}_{1} {L}_{01} \\
 {\mu}_{0}{L}_{10} +{\mu}_{1}{L}_{11}  \\
   {\mu}_{0}{L}_{20} +{\mu}_{1}{L}_{21} \\
\vdots  ~~~~~~~~\vdots  \\
 {\mu}_{0}{L}_{m0} +{\mu}_{1}{L}_{m1}
 \end{bmatrix}}
 \end{eqnarray*}
 where $L_{0j},L_{1j}, \cdots, L_{mj}$ are the coefficients of the Lagrange polynomials  $$L_{j}(x)=L_{0j}+L_{1j}x+ \cdots +L_{mj}x^{m}=\displaystyle \prod_{\substack{0\le i \le m\\ i\neq j}}\frac{(x-\lambda_{i})}{(\lambda_{j}-\lambda_{i})}.$$
 Therefore the required polynomial is,
 \begin{eqnarray*} 
p(x)&=&({\mu}_{0}{L}_{00}+{\mu}_{1} {L}_{01})+ ({\mu}_{0}{L}_{10} +{\mu}_{1}{L}_{11})x + ({\mu}_{0}{L}_{20} +{\mu}_{1}{L}_{21})x^{2}+ \cdots + ({\mu}_{0}{L}_{m0} +{\mu}_{1}{L}_{m1})x^{m} \\
 &=& \mu_{0}({L}_{00}+{L}_{10}x+{L}_{20}x^{2}+\cdots + {L}_{m0}x^{m}) +\mu_{1}({L}_{01}+{L}_{11}x+{L}_{21}x^{2}+\cdots + {L}_{m1}x^{m}) \\ 
 &=& \mu_{0}L_{0}(x) + \mu_{1}L_{1}(x)\\
 &=&\mu_{0}\displaystyle \prod_{i=1}^{m}\frac{(x-\lambda_{i})}{(\lambda_{0}-\lambda_{i})} + \mu_{1}\displaystyle \prod_{\substack{ i =0 \\ i\neq 1}}^{m}\frac{(x-\lambda_{i})}{(\lambda_{1}-\lambda_{i})} \\
 &=& s\displaystyle \prod_{i=1}^{m}\frac{(x-\lambda_{i})}{(\lambda_{0}-\lambda_{i})} -\frac{s}{n-1}\displaystyle \prod_{\substack{ i =0 \\ i\neq 1}}^{m}\frac{(x-\lambda_{i})}{(\lambda_{1}-\lambda_{i})}\\
 &=&s\left\{\displaystyle \prod_{i=1}^m\frac{(x-b_{i}+i)}{(b_{0}-b_{i}+i)}- \frac{1}{n-1} \displaystyle \prod_{\substack{i=0\\ i\neq1}}^m \frac{(x-b_{i}+i)}{(b_{1}-b_{i}+i-1)}\right\}.
  \end{eqnarray*}
  Hence complete the proof.
    \end{proof}
\begin{remark}
    \emph{Although for $n<8,$ the multiplicity of $p(\lambda_{i})$ proof technique is not working. Still it follows $p(\lambda_{0})=\mu_{0},~ p(\lambda_{1})=\mu_{1}$ and $p(\lambda_{i})=\mu_{2}=0$ for $i=2,3.$}
\end{remark}
   
    \begin{theorem}\label{hampoly}
    Let H(d,q) be the Hamming graph with intersection numbers $c_{i}$ and $b_{i}$. Let $A$ and $D$ be the adjacency matrix and distance matrix. Then the polynomial $p(x)=t\left\{\displaystyle \prod_{i=1}^d\frac{(x-b_{0}+qc_{i})}{qc_{i}}- \frac{1}{d(q-1)} \displaystyle \prod_{\substack{i=0\\ i\neq1}}^d \frac{(x-b_{0}+qc_{i})}{q(c_{i}-1)}\right\}$ is the required polynomial satisfy $D=p(A)$, where $t=dq^{d-1}(q-1)$.

\end{theorem}
\begin{proof}
         We know that the intersection numbers of the Hamming graph $H(d,q)$ are $c_{i}=i, b_{i}=(d-i)(q-1).$ The distinct eigenvalues of $A$ are $\lambda_{i}=b_{0}-qc_{i}$ with multiplicity $\binom{d}{i}(q-1)^{i} $ for $i=0,1, \cdots, d.$ The distinct eigenvalues of $D$ are $\mu_{0}=s=dq^{d-1}(q-1), \mu_{1}=-q^{(d-1)}, \mu_{2}=0$ with multiplicity $1, d(q-1), q^{d}-d(q-1)-1$ respectively. For every distance regular graph with diameter $r,$ there exists a polynomial $p(x)$ of degree $r$ such that $p(A)=D.$ So, $p(\lambda_{i}) \in \{ \mu_{0}, \mu_{1},\mu_{2} \}.$ \textit{Claim:} $p(\lambda_{0})=\mu_{0}, p(\lambda_{1})=\mu_{1}$ and $p(\lambda_{i})=\mu_{2}=0$ for $i=2,3, \cdots, d.$ It is obvious that $p(\lambda_{0})=\mu_{0},$ because $p(\lambda_{0})$ is an eigenvalue of the matrix $p(A)$ and $\mu_{0}$ is the eigenvalue of the matrix $D$ corresponding the same eigenvector $j_{q^d}.$ So, mainly we have to show $p(\lambda_{1})=\mu_{1}.$ We know, if $\lambda$ has multiplicity $\alpha$ then $p(\lambda)$ has multiplicity at least $\alpha.$ To prove $p(\lambda_{1})=\mu_{1},$ we prove that $p(\lambda_{i})\neq \mu_{1}$ for $i=2, 3, \cdots, d.$ For this we have to show that all $\lambda_{i},$ for $i=2, 3, \cdots, d;$ have more multiplicity than $\lambda_{1}.$ Basically we have to prove $\binom{d}{i}(q-1)^{i}>d(q-1).$ Now it is obvious that $\binom{d}{i}(q-1)^{i}>\binom{d}{1}(q-1)^{i}>d(q-1)$ for $2 \leq i \leq (d-2).$ For $i=d-1,\binom{d}{d-1}(q-1)^{d-1}=\binom{d}{1}(q-1)^{d-1}>d(q-1).$ For $i=d, \binom{d}{d}(q-1)^{d}=(q-1)^{d}>d(q-1),$ because $(q-1)^{d-1}>d$ for any $q,d\geq 3.$ Hence, for $i=2, 3, \cdots, d$ and $q,d\geq 3;$ each $p(\lambda_{i})$ has multiplicity more than $d(q-1).$ This implies $p(\lambda_{i})\neq \mu_{1}$ for $i=2, 3, \cdots, d.$ Hence $p(\lambda_{0})=\mu_{0}, p(\lambda_{1})=\mu_{1}$ and $p(\lambda_{i})=\mu_{2}=0$ for $i=2,3, \cdots, d.$  Let $p(x)=a_{0} + a_{1}x + a_{2}x^{2} + \cdots + a_{d}x^{d}$ be polynomial of degree $d$ which satisfies $p(\lambda_{0})=\mu_{0}, p(\lambda_{1})=\mu_{1}, p(\lambda_{i})=\mu_{2}=0$ for $i=2, 3, \cdots, d.$ This problem can be viewed in matrix form as follows:
         \begin{eqnarray*}\label{}\displaystyle {\begin{bmatrix}  {1} & {\lambda}_{0} & {\lambda}_{0}^{2} &\cdots &{\lambda}_{0}^{d} \\
 {1} & {\lambda}_{1} & {\lambda}_{1}^{2} &\cdots & {\lambda}_{1}^{d} \\
 {1} & {\lambda}_{2} & {\lambda}_{2}^{2} &\cdots & {\lambda}_{2}^{d} \\
\vdots &\vdots &\vdots &\ddots &\vdots   \\
{1} & {\lambda}_{d}& {\lambda}_{d}^{2}& \cdots & {\lambda}_{d}^{d}
 \end{bmatrix}}
 \displaystyle {\begin{bmatrix}   {a}_{0}  \\
  {a}_{1}  \\
   {a}_{2} \\
\vdots    \\
 {a}_{d}
 \end{bmatrix}}=\displaystyle {\begin{bmatrix}   p({\lambda}_{0})  \\
  p({\lambda}_{1})  \\
   p({\lambda}_{2}) \\
\vdots    \\
 p({\lambda}_{d})
 \end{bmatrix}}=\displaystyle {\begin{bmatrix}   {\mu}_{0}  \\
  {\mu}_{1}  \\
   {0} \\
\vdots    \\
 {0}
 \end{bmatrix}}
 \end{eqnarray*}
 Using equation \ref{vaninv} we get,
 \begin{eqnarray*}
 \displaystyle {\begin{bmatrix}   {a}_{0}  \\
  {a}_{1}  \\
   {a}_{2} \\
\vdots    \\
 {a}_{d}
 \end{bmatrix}}=\displaystyle {\begin{bmatrix}  {L}_{00} & {L}_{01} &\cdots &{L}_{0d} \\
  {L}_{10} & {L}_{11} &\cdots & {L}_{1d} \\
  {L}_{20} & {L}_{21} &\cdots & {L}_{2d} \\
\vdots &\vdots &\ddots &\vdots   \\
 {L}_{d0}& {L}_{d1}& \cdots & {L}_{dd}
 \end{bmatrix}} \displaystyle {\begin{bmatrix}   {\mu}_{0}  \\
  {\mu}_{1}  \\
   {0} \\
\vdots    \\
 {0}
 \end{bmatrix}}=\displaystyle {\begin{bmatrix}   {\mu}_{0}{L}_{00}+{\mu}_{1} {L}_{01} \\
 {\mu}_{0}{L}_{10} +{\mu}_{1}{L}_{11}  \\
   {\mu}_{0}{L}_{20} +{\mu}_{1}{L}_{21} \\
\vdots  ~~~~~~~~\vdots  \\
 {\mu}_{0}{L}_{d0} +{\mu}_{1}{L}_{d1}
 \end{bmatrix}}
 \end{eqnarray*}
 where $L_{0j},L_{1j}, \cdots, L_{dj}$ are the coefficients of the Lagrange polynomials $$L_{j}(x)=L_{0j}+L_{1j}x+ \cdots +L_{dj}x^{d}=\displaystyle \prod_{\substack{0\le i \le d\\ i\neq j}}\frac{(x-\lambda_{i})}{(\lambda_{j}-\lambda_{i})}.$$ Therefore, the required polynomial is,
 \begin{eqnarray*}
  p(x)&=&({\mu}_{0}{L}_{00}+{\mu}_{1} {L}_{01})+ ({\mu}_{0}{L}_{10} +{\mu}_{1}{L}_{11})x + ({\mu}_{0}{L}_{20} +{\mu}_{1}{L}_{21})x^{2}+ \cdots + ({\mu}_{0}{L}_{d0} +{\mu}_{1}{L}_{d1})x^{d} \\
 &=& \mu_{0}({L}_{00}+{L}_{10}x+{L}_{20}x^{2}+\cdots + {L}_{d0}x^{d}) +\mu_{1}({L}_{01}+{L}_{11}x+{L}_{21}x^{2}+\cdots + {L}_{d1}x^{d}) \\ 
 &=& \mu_{0}L_{0}(x) + \mu_{1}L_{1}(x)\\
 &=&\mu_{0}\displaystyle \prod_{i=1}^{d}\frac{(x-\lambda_{i})}{(\lambda_{0}-\lambda_{i})} + \mu_{1}\displaystyle \prod_{\substack{ i =0 \\ i\neq 1}}^{d}\frac{(x-\lambda_{i})}{(\lambda_{1}-\lambda_{i})} \\
 &=&t\displaystyle \prod_{i=1}^{d}\frac{(x-\lambda_{i})}{(\lambda_{0}-\lambda_{i})} -\frac{t}{d(q-1)}\displaystyle \prod_{\substack{ i =0 \\ i\neq 1}}^{d}\frac{(x-\lambda_{i})}{(\lambda_{1}-\lambda_{i})}\\
 &=&t\left\{\displaystyle \prod_{i=1}^d\frac{(x-b_{0}+qc_{i})}{(b_{0}-b_{0}+qc_{i})}- \frac{1}{d(q-1)} \displaystyle \prod_{\substack{i=0\\ i\neq1}}^d \frac{(x-b_{0}+qc_{i})}{(b_{0}-q-b_{0}+qc_{i})}\right\}\\
 &=&t\left\{\displaystyle \prod_{i=1}^d\frac{(x-b_{0}+qc_{i})}{qc_{i}}- \frac{1}{d(q-1)} \displaystyle \prod_{\substack{i=0\\ i\neq1}}^d \frac{(x-b_{0}+qc_{i})}{q(c_{i}-1)}\right\}.
 \end{eqnarray*}
Hence, complete the proof.
 \end{proof}



\section{Full Distance Spectrum of Kronecker Product of Some Graphs}
This section provides all the explicit eigenvalues of some product graphs. Let $\mathcal{G}$ be a simple connected graph and $K_{n}$ be a complete graph. Then the Kronecker product $K_{n} \otimes \mathcal{G} $ is a simple connected graph (by Lemma \ref{connected}). The next two results give us all the distance eigenvalues of $K_{n} \otimes C_{m}$, for cycle graph $C_{m}$ with $m$ vertices.
\begin{theorem} \label{c2m}
Let $K_{n}$ and $C_{n}$ denote the complete and cycle graph with $n$ vertices, respectively. Then the distinct distance eigenvalues of $K_{n} \otimes C_{2m}$ are $2(n+1) + nm^2, 2(n-1) + 4 \cos\left(\frac{2p\pi}{m}\right), 2(n-1) + 4 \cos\left(\frac{(2q-1)\pi}{m}\right) - n \mathrm{cosec}^{2}\left(\frac{(2q-1)\pi}{2m}\right), \left(4\cos\left(\frac{\pi}{m}r\right)-2\right) $ for $p=1, 2, \cdots, (m-1); q=1, 2, \cdots, m;$ and $ r=1, 2, \cdots, (2m-1). $ 
\end{theorem}
\begin{proof}
Let $A$ and $D$ be the adjacency matrix and distance matrix of $C_{2m}$ with index by $\{1, 2, \cdots, 2m \}$. Then the distance matrix of $K_{n} \otimes C_{2m}$ is a matrix of size $2mn\times2mn$ and let $D(K_{n} \otimes C_{2m})$ be the distance matrix of $K_{n} \otimes C_{2m}$ with index by $\pi=\{X_{1}, X_{2}, \cdots, X_{n}\}$ where $X_{i}=\{(i,1),(i,2), \cdots, (i,2m)\}$ for $1\leq i \leq n$ and can be written as 
    \begin{eqnarray*}\displaystyle {D(K_{n} \otimes C_{2m})}&=&{\begin{bmatrix}  {2A + D} &{2I + D} &{2I + D} &\cdots & {2I + D}\\
 {2I + D} &{2A + D} &{2I + D} &\cdots & {2I + D}\\
{2I + D} &{2I + D} &{2A + D} &\cdots & {2I + D}\\
\vdots &\vdots &\vdots &\ddots &\vdots \\
{2I + D} &{2I + D} &{2I + D} &\cdots & {2A + D}
 \end{bmatrix}}.
 \end{eqnarray*}
This $D(K_{n} \otimes C_{2m})$ is a real symmetric block circulant matrix. Then by applying Theorem \ref{bcir}, the eigenvalues of $D(K_{n} \otimes C_{2m})$ are the union of the eigenvalues of the matrices $H_{j}$ of size $2m \times 2m,$ for $j=0, 1, 2, \cdots, (n-1);$\\ where $H_{j}=(2A + D) + 2(2I + D) \displaystyle\sum_{k=1}^{h-1} \cos\left(\frac{2k \pi}{n}j\right) + \left\{ \begin{array}{rr}
        0 ;~~~~~~~~&\mbox{ if $n=2h-1$ }\\
         (2I + D)(-1)^j ;&\mbox{ if $n=2h$}~~~~~~
       \end{array}\right.$\\
Case I: If $n=2h-1$, then $H_{j}=(2A + D) + 2(2I + D) \displaystyle\sum_{k=1}^{h-1} \cos\left(\frac{2k \pi}{n}j\right).$
\begin{eqnarray*}
\mbox{ So,} ~~~~~~~~~~~~H_{0}&=&(2A + D) + 2(h-1)(2I + D)~~~~~~~~~~~~~~~~~~~~~~~~~~~~~~~~~~~~~~~~~~~~~~~~\\
&=&(2A + D) + (n-1)(2I + D)\\
&=& 2(n-1)I + nD + 2A.\\
\mbox{ and} ~~~~~~~~~~~~~H_{j}&=&(2A + D) + 2(2I + D) \displaystyle\sum_{k=1}^{h-1} \cos\left(\frac{2k \pi}{n}j\right);~~ \mbox{for}~ j=1,2, \cdots, (n-1).~~~~~\\
&=&(2A + D) + (2I + D)\left[\frac{\sin((h-\frac{1}{2})\frac{2\pi}{n}j)}{\sin(\frac{\pi}{n}j)} - 1\right]\\
&=&(2A + D) + (2I + D)\left[\frac{\sin((2h-1)\frac{\pi}{n}j)}{\sin(\frac{\pi}{n}j)} - 1\right]\\
&=&(2A + D) + (2I + D)\left[\frac{\sin(\pi j)}{\sin(\frac{\pi}{n}j)} - 1\right];~~~~~~[\because n=2h-1.]\\
&=&(2A + D) - (2I + D)\\
&=&2(A-I).
\end{eqnarray*}
Case II: If $n=2h$, then $H_{j}=(2A + D) + 2(2I + D) \displaystyle\sum_{k=1}^{h-1} \cos\left(\frac{2k \pi}{n}j\right) + (-1)^{j}(2I + D).$
   \begin{eqnarray*}
 \mbox{So,}~~~~~~~ H_{0}&=&(2A + D) + 2(2I + D)(h-1) + (2I + D)\\
 &=&(2A + D) + (2h-1)(2I + D)\\
 &=& (2A + D) + (n-1)(2I + D) \\
 &=& 2(n-1)I + nD + 2A~~~~~~~~~~~~~~~~~~~~~~~~~~~~~~~~~~~~~~~~~~~~~~~~~~~~~
 \end{eqnarray*}
 
 \begin{eqnarray*}
 \mbox{and}~~~~~~ H_{j}&=&(2A + D) + 2(2I + D) \displaystyle\sum_{k=1}^{h-1} \cos\left(\frac{2k \pi}{n}j\right) - (2I + D); ~\mbox{for} ~j= 1, 3, 5, \cdots, (n-1).\\ 
 &=&2(A-I) + (2I+D)\left[\frac{\sin((h-\frac{1}{2})\frac{2\pi}{n}j)}{\sin(\frac{\pi}{n}j)}-1 \right]\\
 &=&2(A-I) + (2I+D)\left[\frac{\sin((2h-1)\frac{\pi}{n}j)}{\sin(\frac{\pi}{n}j)}-1 \right]\\
 &=&2(A-I) + (2I+D)\left[\frac{\sin(\pi j -\frac{\pi}{n}j)}{\sin(\frac{\pi}{n}j)}-1 \right]; ~~~~[\because n=2h.]\\
 &=&2(A-I) + (2I+D)\left[\frac{\sin(\frac{\pi}{n}j)}{\sin(\frac{\pi}{n}j)}-1 \right];~~~~~ [\because j~ \mbox{is odd}.]\\
 &=&2(A-I).\\
 \mbox{Also}~~~~~~~ H_{j}&=&(2A + D) + 2(2I + D) \displaystyle\sum_{k=1}^{h-1} \cos\left(\frac{2k \pi}{n}j\right) + (2I + D); ~\mbox{for}~ j= 2, 4, 6, \cdots, (n-2).\\ 
  &=&(2A+D) +(2I+D)+ (2I+D)\left[\frac{\sin(\pi j -\frac{\pi}{n}j)}{\sin(\frac{\pi}{n}j)}-1 \right]\\
  &=&(2A+D) +(2I+D)+ (2I+D)\left[-\frac{\sin(\frac{\pi}{n}j)}{\sin(\frac{\pi}{n}j)}-1 \right];~~~~~ [\because j ~\mbox{is even.} ]\\
  &=&(2A+D) - (2I+D)\\
  &=&2(A-I).
  \end{eqnarray*}
Therefore, for any $n$, both the cases we get,
\begin{eqnarray*}
H_{0}&=& 2(n-1)I + nD + 2A\\
\mbox{and}~~~~ H_{j}&=&2(A-I); ~\mbox{for}~ j=1, 2, \cdots, (n-1).
\end{eqnarray*}
The eigenvalues of $A$ are $\lambda_{r}=2\cos\left(\frac{2\pi}{2m}r\right)$ for $r=0, 1, 2, \cdots, (2m-1).$
 \begin{eqnarray*}
 \mbox{The eigenvalues of $D$ are }~~~
 \mu_{r}=\left \{ \begin{array}{rl}
        m^2; ~&\mbox{ for $r=0$}\\
        0; ~&\mbox{ for $r$ is even}\\
-\mathrm{cosec}^{2}\left(\frac{\pi}{2m}r\right); ~&\mbox{ for $r$ is odd}.~~~~~~~~~~~~~~~~~~~~~~~~~~~~~~~~~~~~~~~~~~~~~~~~~~~~
       \end{array}\right.
    \end{eqnarray*}
Therefore, by Lamma \ref{lam2}, the eigenvalues of $H_{0}$ are 
\begin{eqnarray*}
\lambda_{r}+\mu_{r}&=&\left\{ \begin{array}{rl}
    2(n-1) + nm^2+4; &\mbox{ for $r=0$}\\
        2(n-1)+4\cos\left(\frac{2\pi}{2m}r\right); &\mbox{ for $r=2, 4, \cdots, (2m-2)$}\\
2(n-1)-n\mathrm{cosec}^{2}\left(\frac{\pi}{2m}r\right) + 4\cos\left(\frac{2\pi}{2m}r\right) ; &\mbox{ for $r=1, 3, \cdots, (2m-1)$}. 
\end{array} \right. 
       \end{eqnarray*}
       More precisely, 
       \begin{eqnarray*}
           \lambda_{0}+\mu_{0}&=&2n+ nm^2+2;\\
           \lambda_{2k}+\mu_{2k}&=& 2(n-1)+4\cos\left(\frac{2k\pi}{m}\right); ~~~~~\mbox{ for $k=1, 2, \cdots, (m-1)$ };\\
           \lambda_{2k-1}+\mu_{2k-1}&=&2(n-1)- n \mathrm{cosec}^{2}\left(\frac{(2k-1)\pi}{2m}\right) + 4\cos\left(\frac{(2k-1)\pi}{m}\right) ; ~ \mbox{ for $k=1, 2, \cdots, m$}.
       \end{eqnarray*}
So, the eigenvalues of $H_{0}$ are $2n+ nm^2+2,$ $ 2(n-1)+4\cos\left(\frac{2p\pi}{m}\right),$ $ 2(n-1)- n \mathrm{cosec}^{2}\left(\frac{(2q-1)\pi}{2m}\right) + 4\cos\left(\frac{(2q-1)\pi}{m}\right)$ for $p=1, 2, \cdots, (m-1)$ and $q=1, 2, \cdots, m.$\\
The eigenvalues of $H_{j}$ are $\left(4\cos\left(\frac{\pi}{m}r\right) -2\right)$ for $r=0, 1, 2, \cdots, (2m-1)$ and $j=1, 2, \cdots, (n-1).$ \\
Hence complete the proof.
\end{proof}

\begin{theorem}\label{c2m+1}
Let $K_{n}$ and $C_{n}$ denote the complete and cycle graphs with $n$ vertices respectively. Then the distinct distance eigenvalues of $K_{n} \otimes C_{2m+1}$ are $2(n + 1) + n(m^2 +m) ,$ $ 2(n-1) + 4\cos\left(\frac{4p\pi}{2m+1}\right) - \frac{n}{4} \sec^2\left(\frac{p\pi}{2m+1}\right),$ $ 2(n-1) + 4\cos\left(\frac{2(2q-1)\pi}{2m+1}\right) - \frac{n}{4} \mathrm{cosec}^{2}\left(\frac{(2q-1)\pi}{2(2m+1)}\right),$ $\left(4\cos\left(\frac{2\pi r}{2m+1}\right)-2\right)$ for $p=1, 2, \cdots , (m-1);$ $q=1, 2, \cdots, m;$ and $r= 0, 1, 2, \cdots, 2m.$ 
\end{theorem}

\begin{proof}
    Let $A$ and $D$ be the adjacency and distance matrix of $C_{2m+1}.$ The distance matrix of $K_{n} \otimes C_{2m+1}$ is a matrix of size $(2m+1)n\times(2m+1)n$ with index by $\pi=\{X_{1}, X_{2}, \cdots, X_{n}\}$ where $X_{i}=\{(i,1),(i,2), \cdots, (i,2m+1)\}$ for $1\leq i \leq n$ and can be written as 
    \begin{eqnarray*}\displaystyle {D(K_{n} \otimes C_{2m+1})}&=&{\begin{bmatrix}  {2A + D} &{2I + D} &{2I + D} &\cdots & {2I + D}\\
 {2I + D} &{2A + D} &{2I + D} &\cdots & {2I + D}\\
{2I + D} &{2I + D} &{2A + D} &\cdots & {2I + D}\\
\vdots &\vdots &\vdots &\ddots &\vdots \\
{2I + D} &{2I + D} &{2I + D} &\cdots & {2A + D}
 \end{bmatrix}}.
 \end{eqnarray*}
 This $D(K_{n} \otimes C_{2m+1})$ is a real symmetric block circulant matrix. Then by applying Theorem \ref{bcir}, the eigenvalues of $D(K_{n} \otimes C_{2m+1})$ are the union of the eigenvalues of the matrices $H_{j}$ of size $(2m+1) \times (2m+1),$ for $j=0, 1, 2, \cdots, (n-1);$\\ where $H_{j}=(2A + D) + 2(2I + D) \displaystyle\sum_{k=1}^{h-1} \cos\left(\frac{2k \pi}{n}j\right) + \left\{ \begin{array}{rr}
        0 ; ~~~&\mbox{ if $n=2h-1$ }\\
(2I + D)(-1)^j ; &\mbox{ if $n=2h$}.~~~~~
 \end{array} \right.  $\\
By using same technique of the Theorem \ref{c2m}, we can get
\begin{eqnarray*}
H_{0}&=& 2(n-1)I + nD + 2A\\
\mbox{and}~~~~ H_{j}&=&2(A-I); ~\mbox{for}~ j=1, 2, \cdots, (n-1).
\end{eqnarray*}
The eigenvalues of $A$ are  $\lambda_{r} = 2\cos\left(\frac{2\pi r}{2m+1}\right) $ for $r = 0, 1, 2, \cdots, 2m.$
\begin{eqnarray*}
\mbox {The eigenvalues of $D$ are} ~~\mu_{r} = \left\{ \begin{array}{rl}
 m^2+ m; &\mbox{ for $r = 0$}\\
        -\frac{1}{4}\sec^2\left(\frac{\pi}{2(2m+1)}r\right); &\mbox{ for $r=2, 4, \cdots, 2m$}\\
         -\frac{1}{4} \mathrm{cosec}^{2}\left(\frac{\pi}{2(2m+1)}r\right); &\mbox{ for $r = 1, 3, \cdots, (2m-1).$ }  
       \end{array}\right.~~~~~~~~~~~~~~~~~~~~~~~~~~~~~~
       \end{eqnarray*}
       Therefore, by using Lemma \ref{lam2}, the eigenvalues of $H_{0}$ are
       \begin{eqnarray*}
\lambda_{r}+\mu_{r} = \left\{ \begin{array}{rl}
       2(n-1)+ n(m^2+m)+4; &\mbox{ for $r = 0$}\\
        2(n-1)-\frac{n}{4}\sec^2\left(\frac{\pi r}{2(2m+1)}\right)+4\cos\left(\frac{2\pi r}{2m+1}\right); &\mbox{ for $r = 2, 4, \cdots, (2m-2)$}\\
         2(n-1)-\frac{n}{4}\mathrm{cosec}^{2}\left(\frac{\pi r}{2(2m+1)}\right) + 4\cos\left(\frac{2\pi r}{2m+1}\right) ; &\mbox{ for $r = 1, 3, \cdots, (2m-1)$}.  
       \end{array} \right.
       \end{eqnarray*}
       More precisely,
\begin{eqnarray*}
\lambda_{0}+\mu_{0}& = &2(n+1) + n(m^2+m);\\
   \lambda_{2k}+\mu_{2k}& = & 2(n-1)-\frac{n}{4}\sec^2\left(\frac{\pi k}{(2m+1)}\right)+4\cos\left(\frac{4\pi k}{2m+1}\right); ~~~~\mbox{ for $k=1, 2, \cdots, (m-1)$ }\\
         \lambda_{2k-1}+\mu_{2k-1}& = &2(n-1)-\frac{n}{4}\mathrm{cosec}^{2}\left(\frac{(2k-1)\pi}{2(2m+1)}\right) + 4\cos\left(\frac{2(2k-1)\pi}{2m+1}\right) ; ~~\mbox{ for $k=1, 2, \cdots, m$} 
\end{eqnarray*}
 So, the eigenvalues of $H_{0}$ are $2(n+1) + n(m^2+m), 2(n-1)-\frac{n}{4}\sec^2\left(\frac{\pi p}{(2m+1)}\right) + 4\cos\left(\frac{4\pi p}{2m+1}\right), \\2(n-1) -\frac{n}{4}\mathrm{cosec}^{2}\left(\frac{(2q-1)\pi}{2(2m+1)}\right) + 4\cos\left(\frac{2(2q-1)\pi}{2m+1}\right)$ for $p=1, 2, \cdots, (m-1)$ and $q=1, 2, \cdots, m.$\\
       The eigenvalues of $H_{j}$ are $\left(4\cos\left(\frac{2\pi r}{2m+1}\right)-2\right)$ for $r= 0, 1, 2, \cdots, 2m;$ and $j=1, 2, \cdots, (n-1).$ \\
       Hence complete the proof.
\end{proof}
\begin{theorem}\label{knkm}
    Let $K_{n}$ be the complete graph of $n$ vertices. Then the distance eigenvalues of $K_{n} \otimes K_{m}$ are $mn+m+n-3, (m-3),(n-3),-3$ with multiplicity $1,(n-1),(m-1),(n-1)(m-1)$ respectively.
\end{theorem}
\begin{proof}
    Let $D$ be the distance matrix of the complete graph $K_{m}.$ The distance matrix of $K_{n}\otimes K_{m}$ is a matrix of size $mn \times mn$ with index by $\pi=\{X_{1}, X_{2}, \cdots, X_{n}\}$ where $X_{i}=\{(i,1),(i,2), \cdots, (i,m)\}$ for $1\leq i \leq n$ and can be written as 
    \begin{eqnarray*}\label{}\displaystyle {D(K_{n} \otimes K_{m})}&=&{\begin{bmatrix}  {2D } &{2I + D} &{2I + D} &\cdots & {2I + D}\\
 {2I + D} &{2D } &{2I + D} &\cdots & {2I + D}\\
{2I + D} &{2I + D} &{2D } &\cdots & {2I + D}\\
\vdots &\vdots &\vdots &\ddots &\vdots \\
{2I + D} &{2I + D} &{2I + D} &\cdots & {2D }
 \end{bmatrix}}.
 \end{eqnarray*}
 Now $D(K_{n} \otimes K_{m})$ is a real symmetric block circulant matrix. Then applying the Theorem \ref{bcir}, the eigenvalues of $D(K_{n} \otimes K_{m})$ are the union of the eigenvalues of the matrices $H_{j}$ of size $m \times m,$ for $j=0, 1, 2, \cdots, (n-1);$\\ where $H_{j}=2D + 2(2I + D) \displaystyle\sum_{k=1}^{h-1} \cos\left(\frac{2k \pi}{n}j\right) + \left\{ \begin{array}{rr}
        0 ; &\mbox{ if $n=2h-1$ }\\
         (2I + D)(-1)^j ; & \mbox{ if $n=2h.$}
       \end{array} \right. $\\
By using same technique of the Theorem \ref{c2m}, we can get
\begin{eqnarray*}
H_{0}&=& 2(n-1)I + (n+1)D \\
\mbox{and}~~ H_{j} &=& D-2I; ~\mbox{for}~ j=1, 2, \cdots, (n-1).
\end{eqnarray*} 
 The eigenvalues of $D$ are $(m-1)$ and $-1$ with multiplicity $(m-1).$
 So, the eigenvalues of $H_{0}$ are $mn+m+n-3$ and $(n-3)$ with multiplicity $(m-1).$
 And the eigenvalues of $H_{j}$ are $(m-3)$ and $-3$ with multiplicity 
 $(m-1)$, for $j= 1, 2, \cdots, (n-1).$
 Hence complete the proof.
\end{proof}
\begin{observation}
    The graph $K_{n} \otimes K_{n}$ is a simple connected regular graph with order $n^2$ and regularity $(n-1)^{2}$. This is a distance regular graph with diameter $2$ (by Theorem \ref{dia}) and intersection array $\{(n-1)^2, 2n-2 ; 1 ,(n-1)(n-2)\}.$ 
If we replace $m$ by $n$ in the above Theorem \ref{knkm}, then we can get a distance spectrum of $K_{n} \otimes K_{n}$ in the following Corollary \ref{corokn}. Whereas $K_{n} \otimes K_{n} \otimes K_{n}$ is a regular graph with diameter $2$ and has $4$ distinct adjacency eigenvalues. In the case of a distance regular graph with diameter $d$, the number of distinct adjacency eigenvalues is exactly $d+1$. So $K_{n} \otimes K_{n} \otimes K_{n}$ is not a distance regular graph.
\end{observation}

\begin{corollary}\label{corokn}
     The distance eigenvalues of $K_{n} \otimes K_{n}$ are ${(n-1)(n+3)}, (n-3),-3$ with multiplicity $1, (2n-2), (n-1)^2$ respectively.
\end{corollary}

\begin{theorem}\label{johndist}
    Let $K_{n}$ be the complete graph and $J(m,r)$ be the Johnson graph. Consider the distinct adjacency eigenvalues of $J(m,r)$ be $\lambda_{i}=(r-i)(m-r-i),$ for $i=0, 1, \cdots, r$ and the distinct distance eigenvalues of $J(m,r)$ be $\mu_{0}=s, \mu_{1}=-\frac{s}{m-1}, \mu_{2}=0$ where $s=\displaystyle \sum_{j=0}^{r} jk_{j}$ and $k_{j}=\binom{r}{j}\binom{m-r}{j}$ for $j=0,1, \cdots, r.$ Then the distance eigenvalues of $K_{n} \otimes J(m,r)$ are $2n-2 + ns +\lambda_{0}, 2n-2-\frac{ns}{m-1}+\lambda_{1}, 2n-2+\lambda_{i}$ for $i=2,3, \cdots, r;$ and $\lambda_{i}-2$ for $i=0,1, \cdots, r.$
     
\end{theorem}
\begin{proof}
Let $A$ and $D$ be the adjacency and distance matrix of $J(m,r).$ The distance matrix of $K_{n} \otimes J(m,r)$ is a matrix of size $n \binom{m}{r}\times n \binom{m}{r}$ and we make a suitable partition of the vertex set $V(K_{n} \otimes J(m,r))$ , $\pi=\{X_{1}, X_{2}, \cdots, X_{n}\}$ where $X_{i}=\left\{(i,1),(i,2), \cdots, \left(i,\binom{m}{r}\right)\right\}$. Then we can get the distance partitioned matrix with index by $\pi$ as follows:
    \begin{eqnarray*}\displaystyle {D(K_{n} \otimes J(m,r))}&=&{\begin{bmatrix}  {A + D} &{2I + D} &{2I + D} &\cdots & {2I + D}\\
 {2I + D} &{A + D} &{2I + D} &\cdots & {2I + D}\\
{2I + D} &{2I + D} &{A + D} &\cdots & {2I + D}\\
\vdots &\vdots &\vdots &\ddots &\vdots \\
{2I + D} &{2I + D} &{2I + D} &\cdots & {A + D}
 \end{bmatrix}}.
 \end{eqnarray*}
  Now $D(K_{n} \otimes J(m,r))$ is a real symmetric block circulant matrix. Then applying the Theorem \ref{bcir}, the eigenvalues of $D(K_{n} \otimes J(m,r))$ are the union of the eigenvalues of the matrices $H_{j}$ of size $\binom{m}{r} \times \binom{m}{r},$ for $j=0, 1, 2, \cdots, (n-1);$\\ where $H_{j}=(A+D) + 2(2I + D) \displaystyle\sum_{k=1}^{h-1} \cos\left(\frac{2k \pi}{n}j\right) + \left\{ \begin{array}{rr}
        0 ; &\mbox{ if $n=2h-1$ }\\
         (2I + D)(-1)^j ; &\mbox{ if $n=2h.$}~~~
\end{array} \right. $\\
By using same computation of the Theorem \ref{c2m}, we can get
\begin{eqnarray*}
H_{0}&=& 2(n-1)I + nD + A \\
\mbox{and}~~~ H_{j}&=&A-2I; ~\mbox{for}~ j=1, 2, \cdots, (n-1).
\end{eqnarray*} 
       
Since there exists a polynomial $p(x)$ (defined in the Theorem \ref{johnpoly}) such that $D=p(A).$ By using Theorem \ref{johnpoly}, we get $H_{0}= 2(n-1)I + n p(A) + A. $ The eigenvalues of $H_{0}$ are $2(n-1)+np(\lambda_{i})+\lambda_{i}$ for $i=0,1, \cdots, r; $ i.e., $2n-2 + ns +\lambda_{0}, 2n-2-\frac{ns}{m-1}+\lambda_{1}, 2n-2+\lambda_{i}$ for $i=2,3, \cdots, r.$
The eigenvalues of $H_{j}=A-2I$ are $\lambda_{i}-2$ for $i=0,1, \cdots,r.$
Hence, complete the proof.
\end{proof}
\begin{remark}
   \emph{The distance eigenvalue of the Johnson graph $J(m,r),$ $\mu_{1}=-\frac{s}{m-1}=-\frac{1}{m-1}\displaystyle \sum_{j=0}^{r}j\binom{r}{j}\binom{m-r}{j}=-\frac{1}{m-1}\displaystyle \sum_{j=0}^{r}r\binom{r-1}{j-1}\binom{m-r}{j}=-\frac{r}{m-1}\binom{m-1}{r}=-\binom{m-2}{r-1}$ is an integer, and $\mu_{0}, \mu_{2}$ are integers also. So the Johnson graph $J(m,r)$ is distance integral.}
\end{remark}

\begin{theorem}\label{hamdist}
Let $K_{n}$ be the complete graph and $H(d,q)$ be the Hamming graph with diameter $d$ and $q^d$ vertices. Consider the distinct adjacency eigenvalues of $H(d,q)$ be $\lambda_{i}=(d-i)(q-1),$ for $i=0, 1, \cdots, d$ and the distinct distance eigenvalues of $H(d,q)$ be $\mu_{0}=t= dq^{d-1}(q-1), \mu_{1}=-q^{d-1}, \mu_{2}=0.$ Then the distance eigenvalues of $K_{n} \otimes H(d,q)$ are $2n-2+nt+\lambda_{0},2n-2-q^{d-1}+\lambda_{1}, 2n-2+\lambda_{i}$ for $i=2,3, \cdots,d;$ and $\lambda_{i}-2$ for $i=0,1,\cdots, d.$ 
\end{theorem}
\begin{proof} 
Let $A$ and $D$ be the adjacency and distance matrix of $H(d,q).$ The distance matrix of $K_{n} \otimes H(d,q)$ is a matrix of size $n {q^d}\times n{q^d}$, and we make a suitable partition of the vertex set $V(K_{n} \otimes H(d,q))$ , $\pi=\{X_{1}, X_{2}, \cdots, X_{n}\}$ where $X_{i}=\{(i,1),(i,2), \cdots, (i,{q^d})\}$. Then we can get the distance partitioned matrix with index by $\pi$ as follows:
    \begin{eqnarray*}\displaystyle {D(K_{n} \otimes H(d,q))}&=&{\begin{bmatrix}  {A + D} &{2I + D} &{2I + D} &\cdots & {2I + D}\\
 {2I + D} &{A + D} &{2I + D} &\cdots & {2I + D}\\
{2I + D} &{2I + D} &{A + D} &\cdots & {2I + D}\\
\vdots &\vdots &\vdots &\ddots &\vdots \\
{2I + D} &{2I + D} &{2I + D} &\cdots & {A + D}
 \end{bmatrix}}.
 \end{eqnarray*}
 Now $D(K_{n} \otimes H(d,q))$ is a real symmetric block circulant matrix. Then applying the Theorem\ref{bcir}, the eigenvalues of $D(K_{n} \otimes H(d,q))$ are the union of the eigenvalues of the matrices $H_{j}$ of size ${q}^{d} \times {q}^{d},$ for $j=0, 1, 2, \cdots, (n-1);$\\ where $H_{j}=(A+D) + 2(2I + D) \displaystyle\sum_{k=1}^{h-1} \cos\left(\frac{2k \pi}{n}j\right) + \left\{ \begin{array}{rr}
        0 ; &\mbox{ if $n=2h-1$ }\\
         (2I + D)(-1)^j ; &\mbox{ if $n=2h.$}~~~~
       \end{array} \right. $\\
By using same computation of the Theorem \ref{c2m}, we can get
\begin{eqnarray*}
H_{0}&=& 2(n-1)I + nD + A \\
\mbox{and}~~~~ H_{j}&=&A-2I; ~\mbox{for}~ j=1, 2, \cdots, (n-1).
\end{eqnarray*}
       
Since there exists a polynomial $p(x)$ (defined in the Theorem \ref{hampoly}) such that $D=p(A).$ By using Theorem \ref{hampoly}, we get $H_{0}= 2(n-1)I + n p(A) + A. $ The eigenvalues of $H_{0}$ are $2(n-1)+np(\lambda_{i})+\lambda_{i}$ for $i=0,1, \cdots, d$; i.e., $2n-2 + nt +\lambda_{0}, 2n-2-q^{d-1}+\lambda_{1}, 2n-2+\lambda_{i}$ for $i=2,3, \cdots, d.$
The eigenvalues of $H_{j}=A-2I$ are $\lambda_{i}-2$ for $i=0,1, \cdots,d.$
Hence, complete the proof.
\end{proof}
\begin{remark}
    \emph{ The two results, Theorem \ref{johndist} and Theorem \ref{hamdist}, give us two infinite families of arbitrary diameter graphs $K_{n} \otimes J(m,r)$ and $K_{n} \otimes H(d,q)$ that are distance integral.}
\end{remark}


\section{Concluding Remarks}

In \cite{indu3}, G. Indulal and R. Balakrishnan found an infinite family of graphs with diameter $3$ that has integral adjacency spectra as well as integral distance spectra. Then they posed an open problem: ``Characterize graphs for which both adjacency and distance spectrum are integral. A weaker problem would be: Find new families of graphs for which the adjacency spectrum and the distance spectrum are both integral." To partially address this problem, we have identified two infinite families of graphs: $K_{n} \otimes J(m,r)$ and $K_{n} \otimes H(d,q)$, these two families have an arbitrary diameter. These graphs have both integral adjacency spectra and integral distance spectra. Since the complete graph $K_{n}$, the Johnson graph $J(m,r)$, and the Hamming graph $H(d,q)$ are adjacency integral, $K_{n} \otimes J(m,r)$ and $K_{n} \otimes H(d,q)$ are adjacency integral as well. In this paper, Theorem \ref{johndist} and Theorem \ref{hamdist} demonstrate that $K_{n} \otimes J(m,r)$ and $K_{n} \otimes H(d,q)$ are distance integrals. Additionally, we refer to the Hamming graph and Johnson graph which have an arbitrary diameter; these two families have both integral adjacency spectra and integral distance spectra.

\section*{Declarations}
\textbf{Conflict of interest:} The authors state that they possess no conflicts of interest.

\bibliographystyle{plain}
\bibliography{Ref}

\end{document}